\documentclass[12pt,a4paper]{article}
\usepackage{amsmath, amsfonts, amsthm, latexsym, amscd}
\usepackage{amsmath}
\usepackage{xcolor,graphicx,hyperref}
\usepackage{amsfonts}
\usepackage{amssymb}
\usepackage[all]{xy}

\usepackage[norefs,nocites]{refcheck}

\allowdisplaybreaks[4]

\newtheorem{theorem}{Theorem}[section]
\newtheorem{proposition}[theorem]{Proposition}
\newtheorem{corollary}[theorem]{Corollary}

\newtheorem{remark}[theorem]{Remark}

\newtheorem{final remark}[theorem]{Final Remark}
\newtheorem{definition}[theorem]{Definition}

\begin{document}
	
\title{An anisotropic summability and mixed sequences}
\author{Jamilson R. Campos, Renato Macedo\thanks{\textcolor{black}{Partially supported by Capes}} \, and Joedson Santos\thanks{\textcolor{black}{Supported by CNPq and by Grant 2019/0014 of Para\'{\i}ba State Research Foundation (FAPESQ)}.\newline 2020 Mathematics Subject Classification: 46B45, 47B10, 47L20, 47B37.\newline Keywords: Operator ideals, summing operators, sequence spaces, mixed summable sequences.}}
\date{}
\maketitle
	
\begin{abstract}
	In this paper we define and study a vector-valued sequence space, called the space of anisotropic $(s,q,r)$-summable sequences, that generalizes the classical space of $(s; q)$-mixed sequences (or mixed $(s; q)$-summable sequences). Furthermore, we define two classes of linear operators involving this new space and one of them generalizes the class of $(s; q) $-mixed linear operators due A. Pietsch. Some characterizations, inclusion results and a Pietsch domination-type theorem are presented for these classes. It is worth mentioning that some of these results are new even in the particular cases of mixed summable sequences and mixed summing operators.
\end{abstract}

\section{Introduction}

In the theory of operator ideals, linear operators between Banach spaces that improve series convergence, which in general can be characterized by the transformation of vector-valued sequences, has an outstanding role. The best-known example of these classes is the ideal of absolutely $p$-summing linear operators, widely studied in the excellent reference \cite{djt}.

Among the classic examples of vector-valued sequence spaces that are part of the scope of this theory, the space of  $(s;q)$-mixed sequences in $X$ (also called mixed $(s; q)$-summable sequences), where $X$ is a Banach space, appears in several works of Defant and Floret (see \cite{DF}), Pietsch (see \cite{Apie}), and Matos (see \cite{mario, M.C.}). Given $0< q \le s \le \infty$ and defining $s(q)'$ by $1/{s(q)'} + 1/{s} = 1/{q}$, we say that a sequence $(x_{j})_{j=1}^{\infty}$ in $X$ is mixed $(s;q)$-summable if 
\begin{equation}\label{defmix}
	x_{j} = \tau_{j}y_{j}\ \forall j \in \mathbb{N}, \text{ with } (\tau_{j})_{j=1}^{\infty} \in \ell_{s(q)'} \text{ and } (y_{j})_{j=1}^{\infty} \in \ell^{w}_{s}(X),
\end{equation} 
where $\ell_{s(q)'}(X)$ and $\ell_{s}^{w}(X)$ denote the spaces of absolutely $s(q)'$-summable sequences and weakly $s$-summable sequences, respectively. The space of all mixed $(s;q)$-summable sequences of $X$ is denoted by $\ell_{m(s;q)}(X)$ and the expression  \[\|(x_{j})_{j=1}^{\infty}\|_{m(s;q)} = \inf \|(\tau_{j})_{j=1}^{\infty}\|_{s(q)'}\|(y_{j})_{j=1}^{\infty}\|_{w,s},\] 
where the infimum is considered for all possible representations as in \eqref{defmix}, defines a norm in $\ell_{m(s;q)}(X)$. In \cite[Proposition 16.4.4.]{DF} it is proved that this space satisfies the following inclusions
\begin{equation*}
	\ell_{q}(X) \subset \ell_{m(s;q)}(X) \subset \ell_{q}^{w}(X)
\end{equation*}
and this paper also presents a study of the mixed operators defined by Pietsch, which are called $(m(s;q);p)$-summing operators. These operators are defined by the following property: A continuous linear operator between Banach spaces $T:U\rightarrow X$ is $(m(s;q);p)$-summing if $(T(x_{j}))_{j=1}^{\infty} \in \ell_{m(s;q)}(X)$, whenever $(x_{j})_{j=1}^{\infty} \in \ell_{p}^{w}(U)$. Also, from \cite[Theorems 20.1.1 and 20.1.4]{DF}, the following sentences are equivalent:
\begin{enumerate}
	\item[(i)]  $T$ is $(m(s;q);p)$-summing;	
	\item[(ii)] There is a constant $C>0$ such that
	$$\|T(x_{j})_{j=1}^{m}\|_{m(s;q)} \leq C\|(x_{j})_{j=1}^{m}\|_{w,p},$$
	for all $m \in \mathbb{N}$ and $x_{1},\ldots,x_{m} \in U$;	
	\item[(iii)] There is a constant $C>0$ such that
	 $$\left(\sum_{i=1}^{n} \left( \sum_{k=1}^{m}|x_{k}^{*}(T(x_{i}))|^{s}\right)^{\frac{q}{s}} \right)^{\frac{1}{q}} \leq C \|(x_{i})_{i=1}^{m}\|_{w,p} \|(x^{*}_{k})_{k=1}^{n}\|_{s},$$
	 for all $m,n \in \mathbb{N}$, $x^{*}_{1},\ldots,x^{*}_{m} \in X^{*}$ and $x_{1},\ldots,x_{n} \in U$.	
\end{enumerate}

Observing the above characterization and bearing in mind all that is known about the theory of summing operators, a natural question arises: If we consider $\|(x^{*}_{k})_{k=1}^{n}\|_{s} \leq 1$, what is the relationship between the left side of the inequality in (iii) and the norm $\|\cdot\|_{m(q;s)}$? The search for an answer to this question was one of our main motivations for this work.

In Section 2 we define the space of all anisotropic $(s,q,r)$-summable sequences, denoted by $\ell_{(s,q,r)}^{A}(X)$, and give a norm that makes it a Banach space. Some inclusion of this new space with the classical sequence spaces are presented, which will motivate the definition of two operator spaces in the subsequent section. Here we answer the question left above proving that the space of mixed $(s; q)$-summable sequences is a particular case of the space of anisotropic $(s,q,r)$-summable sequences, for a certain choice of parameters. Also, we present a Dvoretzky-Rogers-type result for $\ell_{(s,q,r)}^{A}(X)$ and show that the correspondence $X \mapsto \ell_{(s,q,r)} ^ {A}(X)$ defines a sequence class (see \cite{G.R}).

In Section 3 we define classes of anisotropic summing operators and, with the abstract approach of \cite{G.R}, study them from the point of view of the theory of operator ideals. We will also present some characterizations, inclusions and equality results involving these classes. The injectivity of ideals is also studied and Pietsch Domination-type Theorems for these classes are presented.

Let us introduce the notation and definitions needed for this paper. We will consider $X,Y$ and $U$ Banach spaces over $\mathbb{K} = \mathbb{R}$ or $\mathbb{C}$. The closed unit ball of $X$ is denoted by $B_{X}$ and and its topological dual by $X^{*}$. The space of all $X$-valued sequences will be denoted by $X^{\mathbb{N}}$. Given $x \in X$ we will denote by $x\cdot e_{j}$ the sequence whose the $j$-th entry is $x$ and all other entries are zero. If $X= \mathbb{K}$ and $x=1$ we denote $1\cdot e_{j}$ only by $e_{j}$. The symbol $X \stackrel{\textrm{1}}{\hookrightarrow} Y$ means that $X$ is a linear subspace of $Y$ and $\|x\|_{Y} \leq \|x\|_{X}$, for every $x \in X$. By $\mathcal{L}(X;Y)$ we denote the Banach space of all continuous linear operators $T: X \rightarrow Y$ endowed with the usual sup norm. By $\Pi_{q,p}$ we denote the ideal of absolutely $(q;p)$-summing linear operators and if $p = q$ we simply write $\Pi_{p}$. We use the standard notation of the theory of operator ideals as in \cite{djt}. Given the sequence spaces $E^\mathbb{N}$ and $F^\mathbb{N}$ and an operator $u \in \mathcal{L}(E;F)$, we define the (linear) operator induced by $u$ by $\widehat{u}: E^\mathbb{N} \rightarrow F^\mathbb{N},  \widehat{u}((x_{j})_{j=1}^{\infty}) = (u(x_{j}))_{j=1}^{\infty}$. 

We present the classical sequence spaces we shall work with:

\noindent $\bullet$ $\ell_\infty(X)$ = bounded $X$-valued sequences with the sup norm.\\
$\bullet$ $c_{00}(X)$ = eventually null $X$-valued sequences with the sup norm.\\
$\bullet$ $\ell_p(X) $ = absolutely $p$-summable $X$-valued sequences with the usual norm $\|\cdot\|_p$.\\
$\bullet$ $\ell_p^w(X)$ = weakly $p$-summable $X$-valued sequences with the norm
$$\|(x_j)_{j=1}^\infty\|_{w,p} = \sup_{x^* \in B_{X^*}}\|(x^*(x_j))_{j=1}^\infty\|_p. $$

\section{Anisotropic summable sequences}

We start this section introducing a family of sequence spaces, the spaces of anisotropic  $(s,q,r)$-summable sequences, denoted by $\ell^{A}_{(s,q,r)} (X)$. In what follows, $1 \leq s,r,q < \infty$ are real numbers and $X$ is a Banach space.

\begin{definition}  \rm
	Let $q < s$ and $r \leq s$ be real numbers. A sequence $(x_{j})_{j=1}^\infty \in X^{%
		\mathbb{N}}$ is said to be \textit{anisotropic $(s,q,r)$-summable} if 
	\begin{equation*} 
		((x^*_{k}(x_{j}))_{k=1}^\infty)_{j=1}^\infty \in \ell_{q}(\ell_{s}),\ \
		\mathrm{whenever\ \ } (x^*_{k})_{k=1}^\infty \in \ell_r(X^{*}),
w	\end{equation*}
	that is,
	\begin{equation}\label{1}
		\sum_{j=1}^{\infty}\left(\sum_{k=1}^{\infty}|x^{*}_{k}(x_{j})|^{s} \right)^{\frac{q}{s}} < \infty,\ \
		\mathrm{whenever\ \ } (x^*_{k})_{k=1}^\infty \in \ell_r(X^{*}).  
	\end{equation}
\end{definition}

\begin{remark}\label{R10}\rm 
Straightforward calculations show that if $s < r$ then $\ell^{A}_{(s,q,r)}(X)= \{0\}$ and if $s \le q$ we obtain  $\ell^{A }_{(s,q,r)}(X)= \ell^{w}_{q}(X)$. This justifies the choices on the parameters in above definition.
\end{remark}

The expression
\begin{equation*}
	\Vert (x_{j})_{j=1}^{\infty }\Vert _{A(s,q,r)}:=\sup_{(x_{k}^{\ast
		})_{k=1}^{\infty }\in B_{\ell _{r}(X^{\ast })}}\left( \sum_{j=1}^{\infty
	}\left( \sum_{k=1}^{\infty }|x_{k}^{\ast }(x_{j})|^{s}\right) ^{{q}/{%
			s}}\right) ^{1/{q}}  
\end{equation*}
defines a norm on the space $\ell_{(s,q,r)}^{\text{A}}(X)$. Indeed, the finiteness of \eqref{1} can be proved using the Closed Graph Theorem and the norm axioms are easily verified. Moreover, using standard calculations and the fact that $\ell_{q}(X) \stackrel{\textrm{1}}{\hookrightarrow} \ell_{\infty}(X)$, for all $q$ and any
Banach space $X$,  we obtain that $\left( \ell_{(s,q,r)}^{\text{A}}(X);\|\cdot\|_{A(s,q,r)}\right) $ is a Banach space.

The space $\ell^{A}_{(s,q,r)}(X)$ can be placed in a
chain with the spaces of absolutely and weakly summable sequences:

\begin{proposition}\label{2} The chain 
	\begin{equation} \label{3}
		\ell_{q}(X) \stackrel{\textrm{1}}{\hookrightarrow}\ell_{(s,q,r)}^{\text{A}}(X)\stackrel{\textrm{1}}{\hookrightarrow} \ell_q^w(X)
	\end{equation}
is valid, for every Banach space $X$. 
\end{proposition}
\begin{proof}
	Let $(x_{j})_{j=1}^{\infty} \in \ell_{q}(X)$ and $(x_{k}^{*})_{k=1}^{\infty} \in \ell_{r}(X^{*})$. Using the continuity of $x_k^*$ and that $r \le s$ we obtain
	$$\left( \sum_{j=1}^{\infty
	}\left( \sum_{k=1}^{\infty }|x_{k}^{\ast }(x_{j})|^{s}\right) ^{{q}/{%
			s}}\right) ^{1/{q}}  \leq \|(x_{k}^{*})_{k=1}^{\infty}\|_{r} \cdot \|(x_{j})_{j=1}^{\infty}\|_{q}.$$
	So, $\|(x_{j})_{j=1}^{\infty}\|_{A(s,q,r)} \leq \|(x_{j})_{j=1}^{\infty}\|_{q}$. On the other hand, if  $(x_j)_{j=1}^\infty \in \ell_{(s,q,r)}^{A}(X)$ we take the sequence $(x^{*}_{k})_{k=1}^{\infty} = x^* \cdot e_1\in \ell_r(X^{*})$, where $x^*\in B_{X^{*}}$, and calculate 
	$$\left( \sum_{j=1}^\infty \left| x^*(x_j) \right|^q\right)^\frac{1}{q} = \left( \sum_{j=1}^\infty \left(\sum_{k=1}^\infty |x^*_{k}(x_{j})|^{s}\right)^%
	{q/s}\right)^{1/q}  \leq \|(x_{j})_{j=1}^{\infty}\|_{A(s,q,r)},$$
	from which it follows that $\|(x_j)_{j=1}^\infty\|_{w,q} \leq \|(x_{j})_{j=1}^{\infty}\|_{A(s,q,r)}$. 
	
\end{proof}
\begin{remark}\rm \label{7}
	As $c_{00}(X) \subseteq \ell_{q}(X)$ and $\ell_{q}^{w}(X) \stackrel{\textrm{1}}{\hookrightarrow} \ell_{\infty}(X)$, for every Banach space $X$, using the chain in (\ref{3}) we obtain 
	\begin{equation}\label{incseq}
		c_{00}(X) \subseteq \ell^{A}_{(s,q,r)}(X)  \stackrel{\textrm{1}}{\hookrightarrow} \ell_{\infty}(X).
	\end{equation}
	In addition, it is a simple task to show that space $\ell_{(s,q,r)}^{\text{A}}(X)$ contains an isometric copy of $X$ by the inclusion $x \in X \mapsto x \cdot e_j \in \ell_{(s,q,r)}^{\text{A}}(X)$, for any choice of $j \in \mathbb{N}$. 
\end{remark}

The next proposition establishes the relationship between anisotropic sequence spaces through their indices.

\begin{proposition}
	If $r_{2} \leq r_{1}\leq s_{1} \leq s_{2}$, $q_{1} \leq q_{2} < s_{2}$ and $q_{1} < s_{1}$, then for each Banach space $X$,
	\begin{equation*}
		\ell_{(s_1,q_1,r_1)}^{\text{A}}(X) \overset{1}{\hookrightarrow}
		\ell_{ (s_2,q_2,r_2) }^{\text{A}}(X).
	\end{equation*}
\end{proposition}
\begin{proof}
	Let $(x_{j})_{j=1}^{\infty} \in \ell_{(s_1,q_1,r_1)}^{\text{A}}(X^*)$ and $(x^{*}_{k})_{k=1}^{\infty} \in \ell_{r_{1}}(X)$. As $s_{1} \leq s_{2}$ and $q_{1} \leq q_{2}$, it follows that
	\begin{equation*}
		\left(\sum_{j=1}^{\infty}\left(\sum_{k=1}^{\infty} |x^{*}_{k}(x_{j})|^{s_{2}}\right)^{\frac{q_{2}}{s_{2}}}\right)^{\frac{1}{q_{2}}} \leq \left(\sum_{j=1}^{\infty} \left(\sum_{k=1}^{\infty} |x^{*}_{k}(x_{j})|^{s_{1}}\right)^{\frac{q_{1}}{s_{1}}}\right)^{\frac{1}{q_{1}}} < \infty
	\end{equation*}
	and as this is also valid for any $(x^{*}_{k})_{k=1}^{\infty} \in \ell_{r_{2}}(X^*)$, since $\ell_{r_{2}}(X^{*}) \subseteq \ell_{r_{1}}(X^{*})$, we have $(x_{j})_{j=1}^{\infty} \in \ell_{(s_2,q_2,r_2)}^{\text{A}}(X)$. The inequality of norms also follows immediately from the above calculation.
\end{proof}

Before we show the conditions under which the space $\ell_{(s,q,r)}^{\text{A}}(X)$ coincides with the other spaces in the chain \eqref{3}, we present some facts.

\begin{remark}\label{remark1}\rm  (a) Let $x\in X$ and $(x^*_k)_{k=1}^\infty \in \ell_r(X^*)$. As $r \leq s$, it follows that $(x^*_k(x))_{k=1}^\infty \in \ell_s$. So, for every $x^*=(x^*_k)_{k=1}^\infty \in \ell_r(X^*)$, the operator 
		\begin{equation*}
			\Psi_{x^*} :X \rightarrow \ell_{s};\ \Psi_{x^*}(x) :=
			(x^*_k(x))_{k=1}^\infty,
		\end{equation*}
		is well-defined, linear and continuous, with $\|\Psi_{x^*}\| \leq
		\left\| (x_k^*)_{k=1}^\infty\right\|_r$.\\
(b) If $\left( x_j \right)_{j=1}^\infty \in \ell_{s}^{w}(X)$, we can calculate, for all $\left( x_k^*\right)_{k=1}^\infty \in \ell_r(X^*)$, 
\begin{align*}
	\left( \sum_{j=1}^\infty \sum_{k=1}^\infty \left| x_k^*(x_j) \right|^s \right)^\frac{1}{s} & = \left( \sum_{k=1}^\infty \left\| x^*_k \right\|^s\sum_{j=1}^\infty \left|\dfrac{x_k^*}{\left\| x^*_k \right\|}(x_j) \right|^s \right)^\frac{1}{s} \\
	& \leq \left\| (x_j)_{j=1}^\infty \right\|_{w,s} \left\| (x^*_k)_{k=1}^\infty \right\|_s\\
	& \leq \left\| (x_j)_{j=1}^\infty \right\|_{w,s} \left\| (x_k^*)_{k=1}^\infty \right\|_r,
\end{align*}		
and therefore
\begin{equation} \label{des09}
	\left( \sum_{j=1}^\infty \sum_{k=1}^\infty \left| x_k^*(x_j) \right|^s \right)^\frac{1}{s} \leq \left\| (x^*_k)_{k=1}^\infty \right\|_r \left\| (x_j)_{j=1}^\infty \right\|_{w,s},~\forall \left( x_k^*\right)_{k=1}^\infty \in \ell_r(X^*).
\end{equation}	
\end{remark}

\begin{theorem}\label{R24} Let $X$ be a Banach space. Then $\ell_{(s,q,r)}^{\text{A}}(X) =\ell_q^w(X)$ if and only if \[\Psi_{x^*} \in \Pi_{q}(X;\ell_{s}), \ \forall\ x^*=(x^*_{k})_{k=1}^\infty \in \ell_r(X^{*}).\]
\end{theorem}

\begin{proof}
	If $(x_{j})_{j=1}^{\infty} \in \ell_q^w(X) = \ell_{(s,q,r)}^{\text{A}}(X)$ then, for all $x^{*} = (x^{*}_{k})_{k=1}^{\infty} \in \ell_{r}(X^{*})$, we have
	$$\left(\sum_{j=1}^{\infty}\|\Psi_{x^*}(x_j)\|_{s}^{q}\right)^{1/{q}} =\left(\sum_{j=1}^\infty \left(\sum_{k=1}^\infty |x^*_{k}(x_{j})|^{s}\right)^%
	{q/s} \right)^%
	{1/q}< \infty.$$
	Thus, $\Psi_{x^*} \in \Pi_{q}(X;\ell_{s})$, for all $x^*=(x^*_{k})_{k=1}^\infty \in \ell_r(X^{*}).$
	
	Reciprocally, suppose that $\Psi_{x^*} \in \Pi_{q}(X;\ell_{s})$ for all $x^*=(x^*_{k})_{k=1}^\infty \in \ell_r(X^{*})$.  If $(x_{j})_{j=1}^\infty \in \ell_q^w(X)$ it follows that $(\Psi_{x^*}(x_j))_{j=1}^\infty=((x^*_{k}(x_{j}))_{k=1}^\infty)_{j=1}^\infty \in \ell_{q}(\ell_{s})$, for all $x^* = (x^{*}_{k})_{k=1}^{\infty} \in \ell_r(X^{*})$. Thus, we have $(x_{j})_{j=1}^\infty \in \ell_{(s,q,r)}^{\text{A}}(X)$ and therefore $\ell^{w}_{q}(X) \subseteq \ell_{(s,q,r)}^{\text{A}}(X)$. From this and Proposition \ref{2} we obtain $\ell^{w}_{q}(X) = \ell_{(s,q,r)}^{\text{A}}(X)$.
\end{proof}

The next result is a Dvoretzky-Rogers-type theorem for the space of all anisotropic $(s,q,r)$-summable sequences.

\begin{theorem}\label{TDR}
	A Banach space $X$ has finite dimension if and only if $\ell_{(s,q,r)}^{\text{A}}(X)=\ell_q(X)$.
\end{theorem}

\begin{proof}
	If $X$ has finite dimension it is clear that $\ell_{(s,q,r)}^{\text{A}}(X)=\ell_q(X)$, since $\ell_{q} ^{w}(X)=\ell_q(X)$. Reciprocally, let  $(x_j)_{j=1}^\infty$ be a sequence in $\ell_{q}^{w}(X)\backslash\ell_q(X)$. As $q<s$, there is a real number $t$ such that $1/q = 1/s + 1/t$ and, for each sequence $(\lambda_j)_{j=1}^\infty\in \ell_t$, the sequence $(\lambda_{j}x_{j})_{j=1}^\infty\in \ell_{(s,q,r)}^{\text{A}}(X)$. Indeed, we have
	\begin{align*}
		\left( \sum_{j=1}^\infty \left( \sum_{k=1}^\infty \left| x_k^*(\lambda_j x_j) \right|^s \right)^\frac{q}{s}\right)^\frac{1}{q} 
		& = \left( \sum_{j=1}^\infty \left| \lambda_j \right|^q \left( \sum_{k=1}^\infty \left| x_k^*(x_j) \right|^s \right)^\frac{q}{s}\right)^\frac{1}{q}\\
		& \leq \left( \sum_{j=1}^\infty \left| \lambda_j \right|^{t}\right)^\frac{1}{t} \left( \sum_{j=1}^\infty \sum_{k=1}^\infty\left| x_k^*(x_j) \right|^s \right)^\frac{1}{s} \\
		& \leq \left\| (\lambda_j)_{j=1}^\infty \right\|_{t} \left\| (x_j)_{j=1}^\infty \right\|_{w,s} \left\| (x_k^*)_{k=1}^\infty \right\|_r,
	\end{align*}
	whenever $\left( x_k^*\right)_{k=1}^\infty \in \ell_r(X^*)$, where in the last inequality we use \eqref{des09} and that $(x_j)_{j=1}^\infty \in \ell_{s}^{w}(X)$ since $q<s$.
	
	We notice that if $\sum_{j=1}^\infty \left|a_j\right|\left\|x_j\right\|^q < \infty$, for each $(a_j)_{j=1}^\infty\in \ell_\frac{t}{q}$, we must have $\left(\left\|x_j\right\|^q\right)_{j=1}^\infty\in \ell_\frac{s}{q}$, that is, $(x_j)_{j=1}^\infty\in \ell_s(X)$. Therefore, there is a sequence $(a_j)_{j=1}^\infty\in \ell_\frac{t}{q}$ such that
	\begin{equation}\label{diverge}
		\sum_{j=1}^\infty \left|a_j\right|\left\|x_j\right\|^q=\infty.
	\end{equation}
	Taking $(\lambda_j)_{j=1}^\infty:=\left(|a_j|^{1/q}\right)_{j=1}^\infty \in \ell_t$, by the above calculation, we have $(\lambda_{j}x_{j})_{j=1}^\infty\in \ell_{(s,q,r)}^{\text{A}}(X)$ but (\ref{diverge}) ensures that $(\lambda_{j}x_{j})_{j=1}^\infty\notin \ell_{q}(X)$. So, $(\lambda_{j}x_{j})_{j=1}^\infty\in \ell_{(s,q,r)}^{\text{A}}(X)\backslash\ell_q(X)$.	
\end{proof}

Now we prove one of our main results and this one answers the question left in the introduction of the paper. This result makes use of a characterization for the norm $\left\| \cdot \right\|_{m(s;q)}$ which can be found in \cite[Theorem 1.1]{M.C.}. In order to help the reader, we will show this result first.

\begin{theorem}{\rm \cite[Theorem 1.1]{M.C.}}\label{Maurey}
	For $0 < q < s$ and $(x_{j})_{j=1}^{\infty} \in \ell_{\infty}(X)$ the following sentences are equivalent:
\begin{description}
	\item [(i)] $(x_{j})_{j=1}^{\infty}$ is $m(s,q)$-summable.	
	\item [(ii)] If $W(B_{X^{*}})$ denotes the set of all regular probability measures defined on the $\sigma$-algebra of the Borel subsets of $B_{X^{*}}$, when this set is endowed with the restricted weak star topology of $X^{*}$, it follows that
	$$\left(\left(\int_{B_{X^*}}|x^{*}(x_{j})|^{s} d\mu(x^{*})\right)^{\frac{1}{s}}  \right)_{j=1}^{\infty} \in \ell_{q},$$
	for every $\mu \in W(B_{X^{*}})$. In the case, 
	$$\|(x_{j})_{j=1}^{\infty}\|_{m(s,q)} = \sup_{\mu \in W(B_{X^{*}})} \left\| \left(\left(\int_{B_{X^*}}|x^{*}(x_{j})|^{s} d\mu(x^{*})\right)^{\frac{1}{s}}  \right)_{j=1}^{\infty}\right\|_{q}.$$
\end{description}
\end{theorem}

\begin{theorem}\label{R60}
	Let $X$ be a Banach space. If $1\leq q < s < \infty$, then 
	\begin{equation*} 
		\ell_{m(s;q)}(X)= \ell^{A}_{(s,q,s)}(X) \ \text{ and  }
		\left\| (x_j)_{j=1}^\infty \right\|_{m(s;q)} = \left\| (x_j)_{j=1}^\infty \right\|_{A(s,q,s)}.
	\end{equation*} 
\end{theorem}
\begin{proof}
	Let $\left( x_j \right)_{j=1}^\infty = \left( \lambda_ju_j\right)_{j=1}^\infty \in \ell_{m(s;q)}(X)$, with $(\lambda_j)_{j=1}^\infty \in \ell_{s(q)'}$ and $(u_j)_{j=1}^\infty \in \ell_s^w(X)$. Using the H\"{o}lder's inequality, we have 
	\begin{align*}
		\left( \sum_{j=1}^\infty \left( \sum_{k=1}^\infty \left| x_k^*(x_j) \right|^s \right)^\frac{q}{s}\right)^\frac{1}{q} 
		& = \left( \sum_{j=1}^\infty \left( \sum_{k=1}^\infty \left| x_k^*(\lambda_j u_j) \right|^s \right)^\frac{q}{s}\right)^\frac{1}{q} \\
		& = \left( \sum_{j=1}^\infty \left| \lambda_j \right|^q \left( \sum_{k=1}^\infty \left| x_k^*(u_j) \right|^s \right)^\frac{q}{s}\right)^\frac{1}{q} \\
		& \le \left( \sum_{j=1}^\infty \left| \lambda_j \right|^{s(q)'}\right)^\frac{1}{s(q)'} \left( \sum_{j=1}^\infty \sum_{k=1}^\infty\left| x_k^*(u_j) \right|^s \right)^\frac{1}{s} \\
		& \overset{(\ref{des09})}{\leq} \left\| (\lambda_j)_{j=1}^\infty \right\|_{s(q)'} \left\| (u_j)_{j=1}^\infty \right\|_{w,s} \left\| (x_k^*)_{k=1}^\infty \right\|_s.
	\end{align*}
	So, $\left( x_j \right)_{j=1}^\infty \in \ell_{(s,q,s)}^{A}(X)$ and, as the above calculation is valid for any representation of $(x_j)_{j=1}^\infty \in \ell_{m(s;q)}(X)$, we conclude that 
	\begin{align*}
		\left\| (x_j)_{j=1}^\infty \right\|_{A(s,q,s)}
		& \leq  \|(x_j)_{j=1}^\infty \|_{m(s;q)}.
	\end{align*}
	
	Now let $(x_{j})_{j=1}^{\infty} \in  \ell^{A}_{(s,q,s)}(X)$. By the definition of $\|\cdot\|_{A(s,q,s)}$, for all $(y^{*}_{k})_{k=1}^{\infty} \in \ell_{s}(X^{*})$, we have
	\begin{equation*}
		\left(  \sum_{j=1}^{\infty}\left(\sum_{k=1}^{\infty} \left|   y^{*}_{k}(x_{j}) \right|^{s}   \right)^{\frac{q}{s}}  \right)^{\frac{1}{q}} \leq \|(x_{j})_{j=1}^{\infty}\|_{A(s,q,s)}\|(y^*_{k})_{k=1}^{\infty}\|_{s}
	\end{equation*}
	and, in particular, for any $y^{*}_{1},\ldots,y^{*}_{m} \in B_{X^{*}}$ and $n \in \mathbb{N}$,
	\begin{equation}\label{B00}
		\left(  \sum_{j=1}^{n}\left(\sum_{k=1}^{m} \left|   y^{*}_{k}(x_{j}) \right|^{s}   \right)^{\frac{q}{s}}  \right)^{\frac{1}{q}} \leq \|(x_{j})_{j=1}^{n}\|_{A(s,q,s)}\|(y^*_{k})_{k=1}^{m}\|_{s}.
	\end{equation}
	Let us consider $\mu$ a discrete probability measure over $B_{X^{*}}$. So, there are $x_{1}^{*},\ldots,x_{m}^{*} \in B_{X^*}$ such that $\displaystyle\mu = \sum_{k=1}^{m}v_{k}\delta_{k}$, where $\delta_{k}$ is the Dirac associated with $x_{k}^{*}$ and $v_{k}\geq0$ for each $k \in \{1,\ldots,m\}$. Moreover, $\displaystyle\sum_{k=1}^{m}v_{k} = 1$ and it follows  that
	\begin{equation}\label{B01}
		\int_{B_{X^{*}}}|x^{*}(x_{j})|^{s}d\mu(x^{*}) = \sum_{k=1}^{m}|x^{*}_{k}(x_{j})|^{s}v_{k} = \sum_{k=1}^{m}\left| v_{k}^{\frac{1}{s}}x^{*}_{k}(x_{j})\right| ^{s}, ~\forall j \in \mathbb{N},
	\end{equation} 
	and
	\begin{equation*}
		\left\| \left( v_{k}^{\frac{1}{s}}x^{*}_{k}\right) _{k=1}^{m}\right\|_{s} = \left(  \sum_{k=1}^{m} \left\| v_{k}^{\frac{1}{s}}x^{*}_{k} \right\|^{s}\right)^{\frac{1}{s}} 
		=\left(  \sum_{k=1}^{m} v_{k} \left\| x^{*}_{k} \right\|^{s}\right)^{\frac{1}{s}}
		\leq \left(  \sum_{k=1}^{m} v_{k} \right)^{\frac{1}{s}} = 1.
	\end{equation*}
	Therefore, by $(\ref{B00})$ and $(\ref{B01})$, it follows that
	\begin{align*}
		\left(\sum_{j=1}^{n}\left(\int_{B_{X^{*}}} |x^{*}(x_{j})|^{s}d\mu(x^{*}) \right)^{\frac{q}{s}}  \right)^{\frac{1}{q}} &=  	\left(\sum_{j=1}^{n}\left(\sum_{k=1}^{m}\left| v_{k}^{\frac{1}{s}}x^{*}_{k}(x_{j})\right| ^{s} \right)^{\frac{q}{s}}  \right)^{\frac{1}{q}}\\
		&\leq \|(x_{j})_{j=1}^{n}\|_{A(s,q,s)} \cdot \left\| \left( v_{k}^{\frac{1}{s}}x_{k}^{*}\right) _{k=1}^{m}\right\|_{s}\\
		&\leq \|(x_{j})_{j=1}^{n}\|_{A(s,q,s)}.
	\end{align*}
	So, considering $W(B_{X^{*}})$ the set of all regular probability measures defined in the $\sigma$-algebra of the subsets of $B_{X^{*}}$, we can say that
	\begin{equation}\label{A5}
		\left( \sum_{j=1}^{n}\left(\int_{B_{X^*}} \left|   x^*(x_{j})  \right|^{s} d\mu(x^{*})   \right)^{\frac{q}{s}}  \right)^{\frac{1}{q}} \leq  \|(x_{j})_{j=1}^{\infty}\|_{A(s,q,s)} < \infty,
	\end{equation}
	for every discrete probability measure $\mu \in W(B_{X^{*}})$. As these probability measures are dense in $W(B_{X^{*}})$ with respect to the weak topology defined by $C(B_{X^{*}})$, we can conclude that \eqref{A5} is valid for all $\mu \in W(B_{X^{*}})$. Therefore, from Theorem \ref{Maurey}, we have that $\ell^{A}_{(s,q,s)}(X)\subseteq\ell_{\mathfrak{m}(s;q)}(X) $
	and	\begin{equation*}
		\|(x_{j})_{j=1}^{\infty}\|_{m(s;q)} \leq  \|(x_{i})_{j=1}^{\infty}\|_{A(s,q,s)}.
	\end{equation*}
\end{proof}

\begin{remark}\rm (a) The Theorem \ref{R60} is also valid for $q=s$. In fact, in \cite{Apie} Pietsch proves that  $\ell_{m(s;s)}(X)=\ell_{s}^{w}(X)$ and Remark \ref{R10} asserts that $\ell_{s}^{w}(X)= \ell^{A}_{s,s,s}(X)$, with the equality of their respective norms in both cases.\\
(b) Excluding the so-called trivial case (the equality in (a)), Theorem \ref{R60} in conjunction with Theorem \ref{R24} gives us a characterization for the equality between $\ell^{w}_{q}(X)$ and $\ell_{m(q;s)}(X)$, which as far as we know is not present in the literature.\\ 
(c) From Theorems \ref{R60} and \ref{TDR} we obtain a Dvoretzky-Rogers-type theorem for the space of mixed $(s;q)$-summable sequences. However, this result has already been proven in \cite[Theorem 2.1]{M.C.}.
\end{remark}

To end the section, we prove that the correspondence $X \mapsto \ell^{A}_{(s,q,r)}(X)$, denoted by $\ell^{A}_{(s,q,r)}(\cdot)$, is a sequence class. 

\begin{definition}\rm  \cite[Definition 2.1]{G.R} 
	A class of vector-valued sequences $X$, or simply a sequence class $X$, is a rule that assigns to each Banach space $E$ a Banach space $X(E)$ of $E$-valued sequences, that is $X(E)$ is a vector subspace of $E^{\mathbb{N}}$ with the coordinatewise operations, such that:
	$$c_{00}(E) \subseteq X(E) \stackrel{\textrm{1}}{\hookrightarrow} \ell^{\infty}(E)~\text{and}~ \|e_{j}\|_{X(\mathbb{K})} = 1~\text{for every}~ j.$$
\end{definition}

The framework of sequence classes is an abstract theory build in  \cite{G.R} to deal with ideals of linear and multilinear operators characterized by the transformation of
vector-valued sequences. In addition to the construction, results and examples present in paper \cite{G.R} and in the  references \cite{AAR, BotCam2, G.R 2, BCJ, BF, PG, RS} bring us several applications of the sequence classes environment that show its versatility for operator theory. We will (of course) take advantage of this environment in the next section.

For the reader unfamiliar with the sequence classes environment, the definitions used in the following proposition can be found in \cite[Definitions 2.1 and 3.2]{G.R}.

\begin{proposition}\label{SeqClass}
	The correspondence $X \mapsto \ell^{A}_{(s,q,r)}(X)$ is a linearly stable and finitely determined sequence class.
\end{proposition}

\begin{proof}
	Thanks to Remarks \ref{R10} and \ref{7}, to prove that $\ell^{A}_{(s,q,r)}(\cdot)$ is a sequence class it is enough to show that $\|e_{j}\|_{A(s,q,r)} = 1$ for any $j \in \mathbb{N}$. As $e_{j} \in \ell_{(s,q,r)}^{\text{A}}(\mathbb{K})$, by the first inclusion in \eqref{incseq}, and $\|e_{j}\|_{q} = \|e_{j}\|_{w,q} = 1$, the Proposition \ref{2} ensures that $\|e_{j}\|_{A(s,q,r)} = 1$.
	
	Let us prove that $\ell^{A}_{(s,q,r)}(\cdot)$ is finitely determined. As 
	\begin{align*}
		\sup_{n \in \mathbb{N}}\|(x_{j})_{j=1}^{n}\|_{A(s,q,r)} &= \sup_{n \in \mathbb{N}}\sup_{(x^{*}_{k})_{k=1}^{\infty} \in B_{\ell_{r}(X^{*})}}\left(  \sum_{j=1}^{n}\left(\sum_{k=1}^{\infty}|x_{k}^{*}(x_{j})|^{s} \right)^{\frac{q}{s}}\right)^{\frac{1}{q}}\\
		& = \sup_{(x^{*}_{k})_{k=1}^{\infty} \in B_{\ell_{r}(X^{*})}} \sup_{n \in \mathbb{N}} \left(  \sum_{j=1}^{n}\left(\sum_{k=1}^{\infty}|x_{k}^{*}(x_{j})|^{s} \right)^{\frac{q}{s}}\right)^{\frac{1}{q}}\\
		& = \|(x_{j})_{j=1}^{\infty}\|_{A(s,q,r)},
	\end{align*}
we obtain, for $(x_{j})_{j=1}^{\infty} \in X^\mathbb{N}$, that $(x_{j})_{j=1}^{\infty} \in \ell^{A}_{(s,q,r)}(X)$ if and only if $\displaystyle\sup_{n \in \mathbb{N}} \|(x_{j})_{j=1}^{n}\|_{A(s,q,r)} < \infty$ and, in this case,
$$\|(x_{j})_{j=1}^{\infty}\|_{A(s,q,r)} = \sup_{n \in \mathbb{N}} \|(x_{j})_{j=1}^{n}\|_{A(s,q,r)}.$$
This is exactly what it means for $\ell^{A}_{(s,q,r)}(\cdot)$ to be finitely determined.

To prove that $\ell^{A}_{(s,q,r)}(\cdot)$ is linearly stable we need to show that if $(x_{j})_{j=1}^{\infty} \in \ell^{A}_{(s,q,r)}(U)$ and $u \in \mathcal{L}(U;X)$, then $(u(x_{j}))_{j=1}^{\infty} \in \ell^{A}_{(s,q,r)}(X)$ and that, in this case, the (well-defined) induced operator
$\widehat{u}: \ell^{A}_{(s,q,r)}(U) \rightarrow \ell^{A}_{(s,q,r)}(X)$
is continuous with $\|\widehat{u}\| = \|u\|$.

As $\ell_{r}(\cdot)$ is linearly stable (see \cite[Example 3.3]{G.R}), if  $(x^{*}_{k})_{k=1}^{\infty} \in \ell_{r}(X^{*})$ we have  $(x_{k}^{*} \circ u)_{k=1}^{\infty} \in \ell_{r}(U^{*})$. So, if $(x_{j})_{j=1}^{\infty} \in \ell^{A}_{(s,q,r)}(U)$ it follows that
$$\sum_{j=1}^{\infty}\left(\sum_{k=1}^{\infty}|(x_{k}^{*} \circ u)(x_{j})|^{s} \right)^{\frac{q}{s}} < \infty$$
and $(u(x_{j}))_{j=1}^{\infty} \in \ell^{A}_{(s,q,r)}(X)$, proving that the induced operator $\widehat{u}$ it is well defined.

Now, if $(x^{*}_{k})_{k=1}^{\infty} \in B_{\ell_{r}(X^{*})}$ a simple calculations gives us $\left(\frac{x_{k}^{*} \circ u}{\|u\|}\right) _{k=1}^{\infty} \in B_{\ell_{r}(U^{*})}$ and therefore 
\begin{align*}
	\|\widehat{u}\|	&= \displaystyle\sup_{(x_{j})_{j=1}^{\infty}\in B_{\ell^{A}_{(s,q,r)}(U)}}\|
	(u(x_{j}))_{j=1}^{\infty}\|_{A(s,q,r)}\\
	&=\displaystyle\sup_{(x_{j})_{j=1}^{\infty}\in B_{\ell^{A}_{(s,q,r)}(U)}} \left( \sup_{(x^{*}_{k})_{k=1}^{\infty} \in B_{\ell_{r}(X^{*})}} \left(\sum_{j=1}^{\infty}\left(\sum_{k=1}^{\infty}|x^{*}_{k}(u(x_{j}))|^{s} \right)^{q/s}  \right)^{1/q}  \right) \\
	&=\|u\| \displaystyle\sup_{(x_{j})_{j=1}^{\infty}\in B_{\ell^{A}_{(s,q,r)}(U)}} \left( \sup_{(x^{*}_{k})_{k=1}^{\infty} \in B_{\ell_{r}(X^{*})}} \left(\sum_{j=1}^{\infty}\left(\sum_{k=1}^{\infty}\left| \left( \frac{x^{*}_{k} \circ u}{\|u\|}\right)(x_{j}) \right| ^{s} \right)^{q/s}  \right)^{1/q}  \right) \\
	&\leq \|u\| \displaystyle\sup_{(x_{j})_{j=1}^{\infty}\in B_{\ell^{A}_{(s,q,r)}(U)}}\|
	(x_{j})_{j=1}^{\infty}\|_{A(s,q,r)}\leq  \|u\|.
\end{align*}
	On the other hand, if $(x_{j})_{j=1}^{\infty} = x \cdot e_1$, where $x \in B_U$, we know that $\|(x_{j})_{j=1}^{\infty}\|_{A(s,q,r)} \leq 1$. So, for any $x^* \in B_{X^{*}}$, we have $(x_{k}^{*})_{k=1}^{\infty}=x^* \cdot e_1 \in B_{\ell_{r}(X^{*})}$ and  
	\begin{align*}
		|x^*(u(x))| & =  \left( \sum_{k=1}^{\infty }|x^*_{k}(u(x))|^{s}\right) ^{{1}/{s}} = \left( \sum_{j=1}^{\infty
	}\left( \sum_{k=1}^{\infty }|x^*_{k}(u(x_{j}))|^{s}\right) ^{{q}/{
			s}}\right) ^{1/{q}} \\
		& \leq \|(u(x_{j}))_{j=1}^{\infty}\|_{A(s,q,r)}, 
	\end{align*}
	for all $x \in B_{U}$. By the Hahn-Banach Theorem, it follows that
	$$\|u(x)\|= \sup_{x^* \in B_{X^{*}}}|x^*(u(x))| \leq \|(u(x_{j}))_{j=1}^{\infty}\|_{A(s,q,r)} = \|\widehat{u}((x_{j})_{j=1}^{\infty})\|_{A(s,q,r)} \leq \|\widehat{u}\|$$
	for all $x \in B_{U}$ and so  $\|u\| \leq  \|\widehat{u}\|$.	
\end{proof}

To end this session let us comment the impact of Theorem \ref{R60} to our study. 

As we mentioned in Section 1, the norm of the space $\ell_{m(s;q)}(X)$ is originally defined by an infimum and this makes it unfeasible to work in the sequence class environment. For instance, we don't know if $\ell_{m(s;q)}(\cdot)$ is finitely determined. Even the word ``summable'' in the definition of mixed $(s;q)$-summable sequences does not make sense since the expression of $\|\cdot\|_{m(s;q)}$ does not involve a series.

The result of Theorem \ref{R60} actually introduces an alternative norm in the space $\ell_{m(s;q)}(X)$ which takes us to the sequence class environment (through Theorem \ref{SeqClass}) and its entire set of useful tools for the study of summing operators.

\section{Anisotropic summing operators}

We will study some classes of operators that are characterized by transformations of vector-valued sequences and deal with anisotropic summable sequences. In what follows, $1 \leq s,r,q,p < \infty$ are real numbers, $X$ and $U$ are Banach spaces and $T \in \mathcal{L}(U;X)$.

\begin{definition}\label{defweak}\rm 
	Let $1 \le p \leq q < s$ and $1\leq r\le s$ be real numbers. We say that $T$ is \textit{weakly anisotropic $(s,q,r;p)$-summing} if
	\begin{equation*}
	\left(T(u_{j})\right)_{j=1}^{\infty} \in \ell_{(s,q,r)}^{\text{A}}(X)\ \text{ whenever }\ (u_j)_{j=1}^\infty \in \ell_{p}^{w}(U),
	\end{equation*}
	that is, the induced operator $\widehat{T}: \ell_{p}^{w}(U) \rightarrow \ell_{(s,q,r)}^{\text{A}}(X)$ it is well defined.
\end{definition}
\begin{definition}\label{def2}\rm
	Let $1 \le q < s$, $q \leq p$ and  $1\leq r\le s$ be real numbers. We say that $T$ is \textit{anisotropic $(p;s,q,r)$-summing} if
	\begin{equation*}
	\left(T(u_{j})\right)_{j=1}^{\infty} \in \ell_{p}(X) \text{ whenever } (u_j)_{j=1}^\infty \in \ell_{(s,q,r)}^{\text{A}}(U)
	\end{equation*}
	and that means the induced operator $\widehat{T}: \ell_{(s,q,r)}^{\text{A}}(U) \rightarrow \ell_{p}^{}(X)$ it is well defined.
\end{definition}

The symbols $\mathcal{W}^{A}_{(s,q,r;p)}(U ; X )$ and  $\Pi^{A}_{(p;s,q,r)}(U;X)$ denote the classes of all weakly aniso- tropic $(s,q,r;p)$-summing operators and all anisotropic $(p;s,q,r)$-summing operators, from $U$ into $X$, respectively.

\begin{remark}\rm 
	(a) Straightforward calculations show that if $q<p$ then $\mathcal{W}^{A}_{(s,q,r;p)}(U;X)= \{0\}$ and, from Remark \ref{R10}, this also occurs if $s < r$. The same remark ensures that $\mathcal{W}^{A}_{(s,q,r;p)}(U;X) = \mathcal{L}(U;X)$ if $s \le q$. \\
	(b) Again, a simple calculus shows that $\Pi^{A}_{(p;s,q,r)}(U;X) = \{0\}$ if $p<q$ and, for the same reasons as in (a), if $s \leq q$ then $\Pi^{A}_{(p;s,q,r)}(U;X) = \Pi_{p,q}(U;X)$ and $\Pi^{A}_{(p;s,q,r)}(U;X) = \mathcal{L}(U;X)$ if $s<r$.  
\end{remark}

Using Proposition \ref{SeqClass} and the abstract approach of \cite{G.R}, the following two propositions are immediate consequences of \cite[Proposition 2.4]{G.R}.

\begin{proposition} \label{pq}
The following statements are equivalent:
	\begin{description}
		\item[(i)]  $T \in \mathcal{W}^{A}_{(s,q,r;p)}(U;X )$.		
		\item[(ii)] There is a constant $D>0$ such that
		$
		\|\left(  T(u_{j})\right)  _{j=1}^{\infty}\|_{A(s,q,r)} \leq D \|(u_j)_{j=1}^\infty\|_{w,p},
		$
		for every $(u_j)_{j=1}^\infty \in \ell_{p}^{w}(U)$;		
		\item[(iii)] There is a constant $D>0$ such that
		$
		\|\left(  T(u_{j})\right)  _{j=1}^m\|_{A(s,q,r)} \leq D\|(u_j)_{j=1}^m\|_{w,p},
		$
		for every $m \in \mathbb{N}$ and all $u_{1},\ldots,u_{m}$ in $U$.		
		\item[(iv)] There is a constant $D>0$ such that
		\begin{equation*} 
		\left( \sum_{j=1}^\infty \left( \sum_{k=1}^\infty \left| x^*_k\left( T(u_j)\right) \right|^s \right)^{\frac{q}{s}} \right)^{\frac{1}{q}} \leq D \left\| (x^*_k)_{k=1}^\infty \right\|_r \left\| (u_j)_{j=1}^\infty  \right\|_{w,p} ,
		\end{equation*}
		for every $(u_j)_{j=1}^\infty \in \ell_{p}^{w}(U)$ and all $(x^*_k)_{k=1}^\infty \in \ell_r(X^*)$.		
		\item[(v)] There is a constant $D>0$ such that
		\begin{equation*} 
		\left( \sum_{j=1}^m \left( \sum_{k=1}^\infty \left| x^*_k\left( T(u_j)\right) \right|^s \right)^{\frac{q}{s}} \right)^{\frac{1}{q}} \leq D \left\| (x^*_k)_{k=1}^\infty \right\|_r \left\| (u_j)_{j=1}^m  \right\|_{w,p} ,
		\end{equation*}
		for every $m \in \mathbb{N}$, $u_1, \ldots, u_m \in U$ and  all $(x^*_k)_{k=1}^\infty \in \ell_r(X^*)$.		
		\item[(vi)] There is a constant $D>0$ such that
		\begin{equation} \label{mix6}
		\left( \sum_{j=1}^m \left( \sum_{k=1}^n \left| x^*_k\left( T(u_j)\right) \right|^s \right)^{\frac{q}{s}} \right)^{\frac{1}{q}} \leq D \left\| (x^*_k)_{k=1}^n \right\|_r \left\| (u_j)_{j=1}^m  \right\|_{w,p} ,
		\end{equation}
		for every $m,n \in \mathbb{N}$, $u_1, \ldots, u_m \in U$ and $x^*_1, \ldots, x^*_n \in X^*.$
	\end{description}
Furthermore, the infimum of all the constants satisfying \eqref{mix6} (or any other inequality above) defines a norm for  the class $\mathcal{W}^{ A} _{(s,q,r;p)}(U;X)$, which is denoted by $w^{A}_{(s,q,r;p)}(\cdot)$.	
\end{proposition}

\begin{proposition}\label{R12}
	The following statements are equivalent:
	\begin{description}
		\item[(i)] $T \in \Pi^{A}_{(p;s,q,r)}(U;X)$.		
		\item[(ii)] There is a constant $D>0$ such that
		$
		\|(T(u_{j}))_{j=1}^{\infty}\|_{p} \leq D \cdot \|(u_{j})_{j=1}^{\infty}\|_{A(s,q,r)},
		$
		whenever $(u_{j})_{j=1}^{\infty} \in \ell^{A}_{(s,q,r)}(U)$.		
		\item[(iii)] There is a constant $D>0$, such that
		$
		\|(T(u_{j}))_{j=1}^{m}\|_{p} \leq D \cdot \|(u_{j})_{j=1}^{m}\|_{A(s,q,r)},
		$ 
		for all $m \in \mathbb{N}$ and $u_{1},\ldots,u_{m} \in U$.		
	\end{description}
Furthermore, the infimum of all the constants satisfying any inequality above defines a norm for  the class $\Pi^{A} _{(p;s,q,r)}(U;X)$, which is denoted by $\pi^{A}_{(p;s,q,r)}(\cdot)$.	
\end{proposition}

For an operator $T \in \mathcal{L}(U;X^{*})$, where the target space is a dual space, we also have the following characterization. 

\begin{corollary}
	An operator $T \in \mathcal{W}^{A}_{(s,q,r;p)}(U;X^{*})$ if and only if there is a constant $D>0$ such that 
	\begin{equation}\label{R26}
	\left( \sum_{j=1}^\infty \left( \sum_{k=1}^\infty \left|  T(u_j)(x_{k}) \right|^s \right)^{\frac{q}{s}} \right)^{\frac{1}{q}} \leq D \left\| (x_k)_{k=1}^\infty \right\|_r \left\| (u_j)_{j=1}^\infty  \right\|_{w,p} ,
	\end{equation}
	for all  $(u_{j})_{j=1}^{\infty} \in \ell_{p}^{w}(U)$ and  $(x_{k})_{k=1}^{\infty} \in \ell_{r}(X)$, with $w^{A}_{(s,q,r;p)}(T) = \inf\{D: (\ref{R26}) ~\text{holds}\}$. More over, the sentences (v) and (vi) in Proposition \ref{pq} also hold in this case with the necessary adaptations.
\end{corollary}
\begin{proof}
	Let $T \in \mathcal{W}^{A}_{(s,q,r;p)}(U;X^{*})$,  $(u_{j})_{j=1}^{\infty} \in \ell_{p}^{w}(U)$ and  $(x_{k})_{k=1}^{\infty} \in \ell_{r}(X)$. Considering the canonical map $J:X\rightarrow X^{**}$, as $\ell_{r}(\cdot)$ is linearly stable (see \cite[Example 3.3]{G.R}), we have $(J(x_{k}))_{k=1}^{\infty} \in \ell_r(X^{**})$. So, the Proposition \ref{pq} ensures that there is a constant $D>0$ such that
	\begin{align*}
	\left( \sum_{j=1}^\infty \left( \sum_{k=1}^\infty \left|  T(u_j)(x_{k}) \right|^s \right)^{\frac{q}{s}} \right)^{\frac{1}{q}}  &=  \left( \sum_{j=1}^\infty \left( \sum_{k=1}^\infty \left|  J(x_{k})(T(u_{j})) \right|^s \right)^{\frac{q}{s}} \right)^{\frac{1}{q}}\\
	&\leq D \left\| (J(x_k))_{k=1}^\infty \right\|_r \left\| (u_j)_{j=1}^\infty  \right\|_{w,p}\\
	&= D \left\| (x_k)_{k=1}^\infty \right\|_r \left\| (u_j)_{j=1}^\infty  \right\|_{w,p}.
	\end{align*}
	 Reciprocally, let $(u_{j})_{j=1}^{\infty} \in \ell_{p}^{w}(U)$ and $(x^{**}_{k})_{k=1}^{\infty} \in \ell_{r}(X^{**})$. For each $n,m \in \mathbb{N}$, consider the sets $M = \text{span}\{x_{1}^{**},\ldots,x_{n}^{**}\}$ and $N = \text{span}\{T(u_{1}),\ldots,T(u_{m})\}$. Thus, by Principle of Local Reflexibility (\cite[Theorem 8.16]{djt}), for every $\epsilon>0$, there is an operator $R \in \mathcal{L}(M;X)$, with $\|R\| \leq 1+\epsilon$, such that
	$$T(u_{j})(R(x_{k}^{**})) = x^{**}_{k}(T(u_{j})),$$
	for all $k \in \{1\ldots,n\}$ e $j\in \{1\ldots,m\}$. So, we get
	\begin{align*}
	\left( \sum_{j=1}^m \left( \sum_{k=1}^n \left| x_{k}^{**} (T(u_j)) \right|^s \right)^{\frac{q}{s}} \right)^{\frac{1}{q}}  &= \left( \sum_{j=1}^m \left( \sum_{k=1}^n \left| T(u_{j})(R(x_{k}^{**})) \right|^s \right)^{\frac{q}{s}} \right)^{\frac{1}{q}}\\
	&\leq D \left\| (R(x_{k}^{**}))_{k=1}^n \right\|_r \left\| (u_j)_{j=1}^m  \right\|_{w,p}\\
	&\leq D(\epsilon+1) \left\| (x_{k}^{**})_{k=1}^n \right\|_r \left\| (u_j)_{j=1}^m  \right\|_{w,p}.
	\end{align*}
	Now, making $\epsilon \rightarrow 0$ and using that $\ell_{r}(\cdot)$ and $\ell_{p}^{w}(\cdot)$ are finitely determined (see \cite[Example 2.2]{G.R}), we conclude that
	$$	\left( \sum_{j=1}^\infty \left( \sum_{k=1}^\infty \left| x_{k}^{**} (T(u_j)) \right|^s \right)^{\frac{q}{s}} \right)^{\frac{1}{q}} \leq D\left\| (x_{k}^{**})_{k=1}^\infty \right\|_r \left\| (u_j)_{j=1}^\infty  \right\|_{w,p}.$$
	So, from Proposition \ref{pq}, we have $T \in \mathcal{W}^{A}_{(s,q,r;p)}(U;X^{*})$ and $w^{A}_{(s,q,r;p)}(T)= \inf\{D: (\ref{R26}) ~\text{holds}\}$.	
\end{proof}

\begin{remark}\rm\label{piAniso} Some straightforward considerations:\\ 
(a) By the Theorem \ref{R60}, the class of anisotropic $(p;s,q,s)$-summing operators  generalizes the class of $(p; m(s; q))$-summing operators defined by Matos in \cite[Definition 2.2]{M.C.} and the class of weakly anisotropic $(s,q,s;p)$-summing operators  generalizes the class of $(s;q)$-mixed operators defined by Pietsch in \cite[Section 20]{Apie}(weakly anisotropic $(s,q,s;p)$-summing operators with $p=q$).\\
(b) With $p$ and $q$ as in Definition \ref{defweak}, it is immediate that $\Pi_{q,p}(U;X)\subseteq \mathcal{W}^{A}_{(s,q,r;p)}(U;X)$. Also, it follows from Proposition \ref{2} that if $T\in\mathcal{W}^{A}_{(s,q,r;p)}(U;X)$ and $S\in\Pi_{q}(X;Y)$, then $S\circ T\in\Pi_{q,p}(U;Y)$.\\
(c) With $p$ and $q$ as in Definition \ref{def2}, we have $\Pi_{p,q}(U;X)\subseteq\Pi^{A}_{(p;s,q,r)}(U;X)$.
\end{remark} 

Notice that if $p \leq q$ we have $\Pi_{q} \subseteq \Pi_{q,p}$ (this is a particular case of \cite[Theorem 10.4]{djt}). Thus, it follows from Remark \ref{piAniso} (b) that $\Pi_{q} \subseteq\mathcal { W}^{A} _{(s,q,r;p)} $. Now we improve this inclusion by proving that $\Pi_{t}\subseteq\mathcal{W}^{A}_{(s,q,r;p)}$ if $t>q$ and satisfies $1/q = 1 / s + 1/t$. 

\begin{proposition}
	Let $1 \leq q<t < \infty$ satisfying $1/q = 1/s + 1/t$. If $T\in\Pi_{t}(U;X)$, then $T\in\mathcal{W}^{A}_{(s,q,r;p)}(U;X)$.
\end{proposition}

\begin{proof}  Let $T\in\Pi_{t}(U;X)$ and $(u_j)_{j=1}^\infty \in \ell_{p}^{w}(U)\subseteq\ell_{q}^{w}(U)$. So, there are sequences $(\lambda_j)_{j=1}^\infty\in \ell_t$ and $(x_j)_{j=1}^\infty \in \ell_{s}^{w}(X)$ such that $T(u_j)=\lambda_{j}x_{j}$, for all $j \in \mathbb{N}$ (see \cite[Lemma 2.23]{djt}). In the beginning of the proof of the Theorem \ref{TDR}, we saw that under these conditions $(T(u_j))_{j=1}^\infty=(\lambda_{j}x_{j})_{j=1}^\infty\in \ell_{(s,q,r)}^{\text{A}}(X)$. So, from Definition \ref{defweak}, $T\in\mathcal{W}^{A}_{(s,q,r;p)}(U;X)$.
\end{proof}

The next result provides another characterization for weakly anisotropic summing operators using the operator $\Psi_{x^*}$ defined in Remark \ref{remark1}.

\begin{theorem}\label{R34} The following sentences are equivalent:
	\begin{description}
		\item[(i)] $T\in\mathcal{W}^{A}_{(s,q,r;p)}(U;X)$
		\item[(ii)] $\Psi_{x^*}\circ T\in\Pi_{q,p}(U;\ell_s)$, for all $x^*=(x^*_k)_{k=1}^\infty \in B_{\ell_r(X^*)}$
	\end{description}
Moreover, if {\rm (i)} is true (and therefore both sentences), then $w^{A}_{(s,q,r;p)}(T) = \pi_{q,p}\left( \Psi_{x^*}\circ T\right)$.
\end{theorem}
\begin{proof}
	Let $x^*=(x^*_k)_{k=1}^\infty \in B_{\ell_r(X^*)}$ and suppose that $T\in\mathcal{W}^{A}_{(s,q,r;p)}(U;X)$. So, calling Proposition \ref{pq}, we get
	\begin{align*}
	\left(\sum_{j=1}^{\infty}\left\|\Psi_{x^*}\circ T(u_j)\right\|_{s}^{q}\right)^{1/{q}} &= \left(\sum_{j=1}^\infty \left(\sum_{k=1}^\infty |x^*_{k}(T(u_{j}))|^{s}\right)^{q/s} \right)^	{1/q}\\
	& \le w^{A}_{(s,q,r;p)}(T)  \left\| (u_j)_{j=1}^\infty  \right\|_{w,p},
	\end{align*}
	for every sequence $(u_j)_{j=1}^\infty \in \ell_{p}^{w}(U)$. Thus, $\Psi_{x^*}\circ T\in\Pi_{q,p}(U;\ell_s)$ and $\pi_{q,p}\left( \Psi_{x^*}\circ T\right)  \leq w^{A}_{(s,q,r;p)}(T)$. 
	
	Reciprocally, for any $x^* = (x^*_k)_{k=1}^\infty \in \ell_r(X^*)$ and any $(u_j)_{j=1}^\infty \in \ell_{p}^{w}(U)$, we obtain  
	\begin{align*}
	\frac{1}{\left\| (x^*_k)_{k=1}^\infty \right\|_r} \left(\sum_{j=1}^\infty \left(\sum_{k=1}^\infty |x^*_{k}(Tu_{j})|^{s}\right)^%
	{q/s} \right)^%
	{1/q} & = \left(\sum_{j=1}^{\infty}\left\|\Psi_{\frac{\textbf{x}^*}{\|\textbf{x}^*\|_r}}\circ T(u_j)\right\|_{s}^{q}\right)^{1/{q}}\\
	& \leq  \pi_{q,p}\left( \Psi_{\frac{\textbf{x}^*}{\|\textbf{x}^*\|_r}}\circ T\right) \|(u_{j})_{j=1}^{\infty}\|_{w,p},	\end{align*}
	that is, 
	$$\left(\sum_{j=1}^\infty \left(\sum_{k=1}^\infty |x^*_{k}(Tu_{j})|^{s}\right)^%
	{q/s} \right)^%
	{1/q} \leq  \pi_{q,p}\left( \Psi_{\frac{\textbf{x}^*}{\|\textbf{x}^*\|_r}}\circ T\right) \left\| (x^*_k)_{k=1}^\infty \right\|_r \|(u_{j})_{j=1}^{\infty}\|_{w,p}.$$
	The conclusion follows from Proposition $\ref{pq}$ as well as the equality of norms.
\end{proof}

\begin{corollary}{\rm (a)} If $p_{1} \leq q_{1} < s$, $p_{2} \leq q_{2} < s$, $r \leq s$ and $\frac{1}{p_{1}}-\frac{1}{q_{1}} \leq \frac{1}{p_{2}}-\frac{1}{q_{2}}$, then 
	$$W^{A}_{(s,q_{1},r,p_{1})}(U;X) \overset{1}{\hookrightarrow}W^{A}_{(s,q_{2},r,p_{2})}(U;X).$$
{\rm (b)} If $\Psi_{x^*}\in\Pi_{q,p}(X;\ell_s)$, for all $x^* \in B_{\ell_r(X^*)}$, then $\mathcal { W }^{A}_{(s,q,r;p)}(U;X)=\mathcal{L}(U;X)$.
\end{corollary}
\begin{proof} (a) Notice that the indexes that play an important role in the theorem are $p$ and $q$ and that from the hypotheses it is well-known that $\Pi_{q_1,p_1}(U;\ell_s) \subseteq \Pi_{q_2,p_2}(U;\ell_s)$ (see \cite[Theorem 10.4]{djt}). With this, the Theorem \ref{R34} ensures the result. \\
(b) Follows immediately from the ideal property of $\Pi_{q;p}$ that $\Psi_{x^*}\circ T\in\Pi_{q,p}(U;\ell_s)$ for all $T \in \mathcal{L}(U;X)$ and the theorem does the work.
\end{proof}

Certain inclusion and coincidence relationships involving the sequence classes $\ell^{A}_{(s,q,r)}(\cdot)$ and $\ell^{w}_{p}(\cdot)$ give us coincidence results for our classes of anisotropic summing operators in addition to the one already presented above.

\begin{proposition} Let $2 \leq q < \infty$ be a real number. Then:
\begin{description}
	\item[(a)] If $X$ has cotype $q$, then $\ell^{w}_{1}(X)\subseteq\ell^{A}_{(s,q,r)}(X)$.
	\item[(b)] If $X$ or $U$ have cotype $q$, then $\mathcal{L}(U;X)=\mathcal{W}^{A}_{(s,q,r;1)}(U;X)$.
\end{description}	
\end{proposition}
\begin{proof}
	(a) If $X$ has cotype $q$, it follows from \cite[Theorem 11.17]{djt} that the identity operator $id_X:X \rightarrow X$ is $(q,1)$-summing and the Remark \ref{piAniso} (b) ensures that  $id_{X}\in\mathcal{W}^{A}_{(s,q,r;1)}(X;X)$, that is, $\ell^{w}_{1}(X)\subseteq\ell^{A}_{(s,q,r)}(X)$.\\ 
	(b) If $X$ or $U$ have cotype $q$, then (a) ensures that $\ell^{w}_{1}(X)\subseteq\ell^{A}_{(s,q,r)}(X)$ or $\ell^{w}_{1}(U)\subseteq\ell^{A}_{ (s,q,r)}(U)$, i.e. $id_{X}\in\mathcal{W} ^{A}_{(s,q,r;1)}(X;X)$ or $id_{U}\in\mathcal{W}^{A}_{(s,q,r;1)}(U;U)$. The ideal property of $\mathcal{W}^{A}_{(s,q,r;1)}$, guaranteed by Proposition \ref{R21}, proves the result in both cases.
\end{proof}

\begin{proposition}
	The following sentences are equivalent: 
	\begin{description}
		\item[(i)] $\ell^{A}_{(s,q,r)}(U) = \ell^{w}_{q}(U)$;		
		\item[(ii)] $\Pi^{A}_{(q;s,q,r)}(U;X) = \Pi_{q}(U;X)$, for all Banach space $X$;		
		\item[(iii)] $\Pi^{A}_{(q;s,q,r)}(U;\ell_{s}) = \Pi_{q}(U;\ell_{s})$.
	\end{description}
\end{proposition}
\begin{proof}	
	(i) $\Rightarrow$ (ii) If $\ell^{A}_{(s,q,r)}(U) = \ell^{w}_{q}(U)$, we have
	\begin{align*}
		T \in \Pi^{A}_{(q;s,q,r)}(U;X) &\Leftrightarrow (T(u_{j}))_{j=1}^{\infty} \in \ell_{q}(X), ~\text{whenever}~ (u_{j})_{j=1}^{\infty} \in \ell^{A}_{(s,q,r)}(U) = \ell^{w}_{q}(U)
	\end{align*}
	and so $\Pi^{A}_{(q;s,q,r)}(U;X) = \Pi_{q}(U;X)$, for all Banach space $X$.\\
	(ii) $\Rightarrow$ (iii). Just take  $X= \ell_{s}$. \\
	(iii) $\Rightarrow$ (i). If $x^{*}=(x^{*}_{k})_{k=1}^{\infty} \in \ell_{r}(U^{*})$, we have $\Psi_{x^*} \in \Pi^{A}_{(q;s,q,r)}(U;\ell_{s})$, since  $$(u_{j})_{j=1}^{\infty} \in \ell^{A}_{(s,q,r)}(U) \Rightarrow 
	(\Psi_{x^*}(u_{j}))_{j=1}^{\infty} = ((x^{*}_{k}(u_{j})_{k=1}^{\infty}))_{j=1}^{\infty} \in \ell_{q}(\ell_{s}).$$ 
	Then, the hypothesis tells us that $\Psi_{x^*} \in \Pi_{q}(U;\ell_{s})$ and the Theorem \ref{R24} ensures that  $\ell^{A}_{(s,q,r)}(U) = \ell^{w}_{q}(U)$.	
\end{proof}

Now we present two Pietsch domination-type theorems for the classes of anisotropic summing operators. The abstract approach present in \cite{ps} will be used for this purpose and let us lay out a preparation for using it.

Let $X$, $Y$ and $E$ be (arbitrary) non-void sets, $\mathcal{H}$ be a non-void family of mappings from $X$ to $Y$, $G$ be a Banach space and $K$ be a compact Hausdorff topological space. Let
$$R: K \times E \times G \longrightarrow [0,\infty) ~\text{and}~ S: \mathcal{H} \times E \times G \longrightarrow [0,\infty)$$
an arbitrary applications such that
$$\begin{array}{cccc}
R_{x,b} \ : & \! K  & \! \longrightarrow
& \! [0,\infty) \\
& \! \varphi & \! \longmapsto
& \! R_{x,b}(\varphi) = R(\varphi,x,b)
\end{array}
$$
is continuous. If $p \in (0,\infty)$, we say that $f \in \mathcal{H}$ is $RS$-abstract $p$-summing if there is a constant $C>0$ such that
\begin{equation}\label{R43}
\left(\sum_{j=1}^{m}S(f,x_{j},b_{j})^{p} \right)^{\frac{1}{p}} \leq C \sup_{\varphi \in K}  \left(\sum_{j=1}^{m}R(\varphi,x_{j},b_{j})^{p} \right)^{\frac{1}{p}},
\end{equation}
for all $m \in \mathbb{N}$, $x_{1},\ldots,x_{m} \in E$ and $b_{1},\ldots,b_{m} \in G$. The infimum of all the constants $C$ satisfying \eqref{R43} is denoted by $\pi_{RS,p}(f)$. Under these conditions, we can state the following result:

\begin{theorem}{\rm \cite[Theorem 3.1]{ps}}\label{R46}
	A map $f \in \mathcal{H}$ is RS-abstract $p$-summing if and only if there are a constant $C>0$ and a regular probability measure $\mu$ in the borelians of $K $ such that
	$$S(f,x,b) \leq C\left(\int_{K}R(\varphi,x,b)^{p}d\mu(\varphi) \right)^{\frac{1}{p}},$$ 
	for all $x \in E$ and all $b \in G$. Also,  $\pi_{RS,p}(f) = \inf\{C: C  \text{ {\rm satisfies the above inequality}}\}$.
\end{theorem}

At this point we are able to show our results.

\begin{theorem} Let $U$ and $X$ be Banach spaces and $T \in \mathcal{L}(U;X)$. The operator $T$ is weakly anisotropic $(s,p,r;p)$-summing if and only if there are a positive constant $C$ and a regular probability measure $\mu$ in the borelians of $B_{U ^{*}}$, with the weak star topology, such that
	\begin{equation} \label{KT}
	\|\Psi_{x^*}(T(u))\|_{s} \leq C\left(
	{\displaystyle\int\limits_{B_{U^{*}}}}
	\left\vert \varphi(u)\right\vert ^{p}d\mu(\varphi)\right)  ^{\frac{1}{p}}, 
	\end{equation}
	for all $u \in U$ and any $x^*=(x^*_k)_{k=1}^\infty \in B_{\ell_r(X^*)}$. Also, $w^{A}_{(s,p,r;p)}\left(T\right) = \inf\{C: C \text{ {\rm satisfies} \eqref{KT}}\}$.
\end{theorem}
\begin{proof}
	Initially, consider
	$\mathcal{H} = \mathcal{L}(U;X),~ E=U,~ G=B_{\ell_{r}(X^{*})}$ and $K=B_{U^{*}}$,
	where $K$ is equipped with the weak star topology. Also, if $T \in \mathcal{L}(U;X),$ $u \in U$, $x^{*} = (x^{*}_{k})_ {k=1}^{\infty} \in B_{\ell_{r}(X^{*})}$ and $\varphi \in B_{U^{*}}$, we can define
	$$S(T,u,x^{*}) = \|\Psi_{x^*}(T(u))\|_{s}$$
	and
	$$R(\varphi,u,x^{*}) =  | \varphi(u)|.$$
	It is immediate to verify that the application $R_{u,x^{*}}$ is continuous and so $T$ is $RS$-abstract $p$-summing if and only if $T$ is weakly anisotropic $(s,p,r;p)$-summing. The Theorem \ref{R46} ensures the result. 	
\end{proof}

In what follows we use this well-known duality result: the spaces $\ell_{r}(U^*)$ and $(\ell_{r^{*}}(U))^{*}$, with $1=1/r + 1/r^{*}$,  are isometrically isomorphic by the application $x^{*} = (x_{k}^{*})_{k=1}^{\infty} \in \ell_{r}(U^*) \mapsto \varphi_{x^{*}} \in (\ell_{r^{*}}(U))^{*}$, where $\varphi_{x^{*}}((u_{j})_{j=1}^{\infty}) = \sum_{j=1}^{\infty}x_{j}^{*}(u_{j})$, for all $(u_{j})_{j=1}^{\infty} \in \ell_{r^{*}}(U)$. 

\begin{theorem}\label{TeoDomPiet} 
	Let $U$ and $X$ be Banach spaces and $T \in \mathcal{L}(U;X)$. The operator $T$ is anisotropic $(p;s,p,r)$-summing if and only if there are a positive constant $C$ and a regular probability measure $\mu$ in the borelians of $B_{(\ell_{r^{*}}(U))^{*}}$, with the weak star topology, such that
	\begin{align}\label{KL2}
	\|T(u)\| \leq C \left(\int_{B_{(\ell_{r^{*}}(U))^{*}}} \left( \sum_{k=1}^{\infty} |\varphi_{ x^{*}}(u\cdot e_{k})|^{s}\right)^{\frac{p}{s}}  d\mu(\varphi_{ x^{*}})  \right)^{\frac{1}{p}},
	\end{align}
	for every $u \in U$. Furthermore, $\pi^{A}_{(p;s,p,r)}\left(
	T\right)  = \inf\{C: C  \text{ {\rm satisfies} }  \eqref{KL2}\}$.
\end{theorem}
\begin{proof} Using the Proposition \ref{R12} and the isomorphism commented above, we have $T \in \Pi_{(p;s,p,r)}(U;X)$ if and only if there is a constant $C$ such that
	\begin{align*}
		\left(\sum_{j=1}^{\infty}|T(u_{j})|^{p} \right)^{\frac{1}{p}} & \leq C \sup_{(x_{k}^{*})_{k=1}^{\infty} \in B_{\ell_{r}(U^{*})}} \left(\sum_{j=1}^{\infty}\left(\sum_{k=1}^{\infty}|x_{k}^{*}(u_{j}) |^{s}\right)^{\frac{p}{s} } \right)^{\frac{1}{p}}\\
		&=C\sup_{\varphi_{x^{*}} \in B_{(\ell_{r^{*}}(U))^{*}}} \left(\sum_{j=1}^{\infty}\left(\sum_{k=1}^{\infty}|\varphi_{ x^{*}}(u_{j}\cdot e_{k}) |^{s}\right)^{\frac{p}{s} } \right)^{\frac{1}{p}}
	\end{align*}
	So, consider $\mathcal{H} = \mathcal{L}(U;X), ~ E=U, ~ G=\{0\}$ and $K = (B_{\ell_{r^{*}}(U)})^{*}$,
	where $K$ is equipped with the weak star topology. Also, for any $T \in \mathcal{L}(U;X)$, for all $u \in U$ and all $ x^{*}=(x^*_k)_{k=1}^\infty \in \ell_r(U^*)$, consider
	$$S(T,u,0) = \|T(u)\| ~\text{and}~ R(\varphi_{ x^{*}},u,0)=\displaystyle\left(\sum_{k=1}^{\infty}|\varphi_{ x^{*}}(u\cdot e_{k})|^{s} \right)^{\frac{1}{s}}.$$ 
	It is a simple task to prove that $R_{u,0}$ is continuous, for all $u \in U$, and it follows that $T$ is anisotropic $(p;s,p,r)$-summing if and only if is $RS$-abstract $p$-summing.  The result is assured by the Theorem \ref{R46}.	
\end{proof}

Thanks to the monotonicity of the $L_{p}$ norms, we obtain an immediate consequence of Theorem \ref{TeoDomPiet}.

	\begin{corollary}
	If $1 \leq p_{1} \leq p_{2} < \infty $, then $\Pi_{(p_{1};s,p_{1},r)}(U;X) \overset{1}{\hookrightarrow} \Pi_{(p_{2};s,p_{2},r)}(U;X)$.
	\end{corollary}

To end the paper we show that the classes of anisotropic summing operators fit into the operator ideal structure and in addition they have the property of injectivity of ideals.

\begin{proposition}\label{R21}
	The pairs $\left( \mathcal{W}^{A}_{(s,q,r;p)};w^{A}_{(s,q,r;p)}(\cdot)\right)$ and $\left( \Pi^{A}_{(p;s,q,r)};\pi^{A}_{(p;s,q,r)}(\cdot)\right)$ are Banach ideals of operators.
\end{proposition}
\begin{proof}
	From Theorem \ref{TDR}, we have $\ell_{q} = \ell^{A}_{(s,q,r)}(\mathbb{K})$ and, in the case where $p \leq q$, we get
	$$\ell_{p}^{w}(\mathbb{K}) = \ell_{p} \overset{1}{\hookrightarrow} \ell_{q} = \ell^{A}_{(s, q,r)}(\mathbb{K}).$$
	On the other hand, in the case $q \leq p$, it follows that 
	$$\ell^{A}_{(s,q,r)}(\mathbb{K})=\ell_ {q} \stackrel{\textrm{1}}{\hookrightarrow} \ell_{p}(\mathbb{K}).$$
	Also, by Proposition \ref{SeqClass} and \cite[Example 3.3]{G.R}, $\ell^{A}_{(s,q,r)}(\cdot)$, $\ell^{w} _ {p}(\cdot)$ and $\ell_{p}(\cdot)$ are linearly stable sequence classes. So it follows from \cite [Theorem 3.6]{G.R} that $\left(\mathcal{W}^{A}_{(s,q,r; p )}; w^{A} _{(s, q ,r;p)}(\cdot)\right)$ and $\left( \Pi^{A}_{(p;s,q,r)};\pi^{A}_{(p;s,q,r)}(\cdot)\right)$ are Banach ideals.
\end{proof}

Recall that an operator ideal $\mathcal{ I }$ is \textit{injective} if $T \in\mathcal{ I }(U;X)$ whenever $v \in \mathcal{L}(X;Y)$ is a metric injection ($\|v(x)\| = \|x\|$ for every $x \in X$) such that $v \circ T \in \mathcal{ I }(U;Y)$.

\begin{proposition}
	The ideals $\mathcal{W}^{A}_{(s,q,r;p)}$ and $\Pi^{A}_{(p;s,q,r)}$ are injective.
\end{proposition}
\begin{proof}
	An immediate calculus proves the injective property of the ideal $\Pi^{A}_{(p;s,q,r)}$ just using that $\|v \circ T (u_{j})\|=\| T (u_{j})\|$, all $j \in  \mathbb{N}$ and all $(u_{j})_{j=1}^{\infty} \in \ell^{A} _{(s,q,r)}(U)$, and the definition of the norm $\|\cdot\|_p$.
	
	Let us suppose that $v  \circ T \in \mathcal{W}^{A}_{(s,q,r;p)}(U;Y)$. Consider  $(x^*_k)_{k=1}^\infty \in \ell_r(X^*)$ and, for each $k \in \mathbb{N}$, let
	$$
	\begin{array}{cccc}
		z^{*}_{k} \ : & \! Im(v) & \! \longrightarrow
		& \! \mathbb{K} \\
		& \! v(x) & \! \longmapsto
		& \! z^{*}_{k}(v(x)) = x^{*}_{k}(x)
	\end{array}
	$$         
	The injectivity of $v$ ensures the well-definition of $z^{*}_{k}$ and as $x^{*}_{k}$ and $v$ are linear and continuous, we have $z^{*}_{k} \in (Im(v))^{*}$. So, by the Hahn-Banach Theorem, there is $y^{*}_{k} \in Y^{*}$ such that $y_{k}^{*}|_{Im(v)} = z^{*}_{k}$ and $\|y_{k}^{*} \|= \|z_{k}^{*}\|$.  Furthermore, 
	$$\|y^{*}_{k}\|= \|z^{*}_{k}\| = \sup_{y \in B_{Im(v)}}|z^{*}_{k}(y)| = \sup_{\|v(x)\| \leq 1}|z^{*}_{k}(v(x))| =\sup_{x \in B_{X}}|x^{*}_{k}(x)| = \|x^{*}_{k}\|$$
	and so $(y^*_k)_{k=1}^\infty \in \ell_r(Y^*)$. As $v\circ T\in\mathcal{W}^{A}_{(s,q,r;p)}(U;Y)$, it follows that
	\begin{align*}
		\left(\sum_{j=1}^\infty \left(\sum_{k=1}^\infty |x^*_{k}(Tu_{j})|^{s}\right)^%
		{q/s} \right)^%
		{1/q}
		& = \left(\sum_{j=1}^\infty \left(\sum_{k=1}^\infty |z^*_{k}( v(Tu_{j}))|^{s}\right)^%
		{q/s} \right)^%
		{1/q}  \\
		& = \left(\sum_{j=1}^\infty \left(\sum_{k=1}^\infty |y^*_{k}(v(Tu_{j}))|^{s}\right)^%
		{q/s} \right)^%
		{1/q}\\
		& \leq  w^{A}_{(s,q,r;p)}(v\circ T) \left\| (u_j)_{j=1}^\infty  \right\|_{w,p} \left\| (y^*_k)_{k=1}^\infty \right\|_r 
	\end{align*}
	and we conclude that  $T\in\mathcal{W}^{A}_{(s,q,r;p)}(U;X)$.	
\end{proof}

\bigskip

\noindent Departamento de Ci\^{e}ncias Exatas\\
Universidade Federal da Para\'iba\\
58.297-000 -- Rio Tinto -- Brazil\\
e-mail: jamilson@dcx.ufpb.br and jamilsonrc@gmail.com\\

\noindent Departamento de Matem\'atica\\
Universidade Federal da Para\'iba\\
58.051-900 -- Jo\~ao Pessoa -- Brazil\\
e-mail: renato.burity@academico.ufpb.br\\

\noindent Departamento de Matem\'atica\\
Universidade Federal da Para\'iba\\
58.051-900 -- Jo\~ao Pessoa -- Brazil\\
e-mail: joedson.santos@academico.ufpb.br
\end{document}